\documentclass[11pt, a4paper]{amsart}

\usepackage[foot]{amsaddr}  

\usepackage[margin=2.54cm]{geometry}
\usepackage{fancyhdr}

\usepackage{amsmath, amssymb, amsthm}
\usepackage{array}
\usepackage{latexsym}
\usepackage{nicematrix}
\usepackage{bigstrut}
\usepackage{enumerate}
\usepackage{paralist}

\usepackage{mathrsfs}

\usepackage{float}   
\usepackage{subcaption}  
\usepackage{pgf,tikz}
\usetikzlibrary{decorations.pathreplacing,calligraphy}
\usepackage{graphicx}
\usepackage[all]{xy}
\usetikzlibrary{arrows}
\usetikzlibrary{arrows.meta}

\usepackage{url}
\usepackage[colorlinks=true,citecolor=cyan,backref=page]{hyperref}
\usepackage[capitalise]{cleveref}

\usepackage{fdsymbol}
\usepackage{todonotes}
\setuptodonotes{size=\tiny}

\makeatletter
\define@key{todonotes}{AS}[]{%
	\setkeys{todonotes}{color=blue!50}}%

\makeatother

\newcommand{\defn}[1]{{\bf#1}}

\renewcommand{\int}{{\mathord{\hookrightarrow}}}

\newtheorem{theorem}{Theorem}

\newtheorem{lemma}[theorem]{Lemma}

\theoremstyle{definition}

\newtheorem{remark}[theorem]{Remark}

\usepackage{todonotes}

\title{Shellability of semidistributive lattices}

\author[S.~Chen \&  A.~Segovia]{Shuo~Chen$^{1}$ \and Adrien~Segovia$^{2}$}

\address{$^{1}$Institute of Theoretical Physics, Chinese Academy of Sciences}
\email{chenshuo223@mails.ucas.ac.cn}

\address{$^{2}$LACIM, Universit\'e du Qu\'ebec \`a Montr\'eal, Canada}
\email{adrien.segovia@gmail.com}

\begin{document}

\begin{abstract}
We prove that join-semidistributive shellable lattices are join-extremal. In particular, every semidistributive shellable lattice is extremal. Together with known results, this gives the equivalence between extremality, left modularity, EL-shellability and shellability for semidistributive lattices. This answers open questions of M\"uhle and Segovia. 
\end{abstract}
	
\maketitle

\section{Introduction}

\textit{In this article, all posets and lattices are assumed to be finite.}

Shellability of simplicial complexes has a long and fruitful history. We refer the reader to the papers of Bj\"orner and Wachs \cite{bjorner1983lexicographically,bjorner1996shellable,bjorner1997shellable} for details related to the following topological notions. A \defn{chain} of a poset $P$ is a totally ordered subset, and a chain is \defn{maximal} if it is not a proper subset of a larger chain. The \defn{order complex} $\Delta(P)$ of $P$ is the simplicial complex on vertex set $P$ whose faces are the chains of $P$. Thus, the facets of $\Delta(P)$ are the maximal chains of $P$. A simplicial complex is \defn{shellable} if there exists a linear order of its facets $F_1,F_2,\dots,F_k$ such that for all $1\leq i<j\leq k$, there exists $l<j$ such that $F_i\,\cap F_j \subseteq F_l\,\cap F_j$ and $|F_l\,\cap F_j|=|F_j|-1$. We call such a linear order on the facets a \defn{shelling}. We say that a poset $P$ is shellable if $\Delta(P)$ is shellable. Shellability is a useful topological property, but it can be difficult to establish:
determining whether a given pure simplicial complex is shellable is NP-complete
\cite{goaoc2019shellability}.

Many posets arising in algebraic combinatorics are lattices, often endowed with additional structure. A poset $L$ is a \defn{lattice} if every pair $\{x,y\}\subseteq L$ has a greatest lower bound, called the \defn{meet} and denoted by $x\wedge y$, and a least upper bound, called the \defn{join} and denoted by $x\vee y$. The main result of this paper reduces the shellability of an important class of lattices, namely semidistributive lattices, to a considerably simpler property to verify: extremality.

A lattice $L$ is \defn{join-semidistributive} if, for all $x,y,z\in L$, the equality $x\vee y=x\vee z$ implies $x\vee (y\wedge z)=x\vee y$.
The notion of a \defn{meet-semidistributive} lattice is defined dually, and $L$ is \defn{semidistributive} if it is both join-semidistributive and meet-semidistributive. Semidistributive lattices form a widely studied class of lattices that generalizes distributive lattices. Important examples of semidistributive lattices include the weak orders on finite Coxeter groups and lattices of torsion classes of finite-dimensional algebras.

An element $p\in L$ is \defn{join-irreducible} if it covers a unique element, denoted by $p_*$. The set of join-irreducibles of $L$ is denoted by $\mathrm{JIrr}(L)$. Dually, the \defn{meet-irreducibles} are the elements that are covered by a unique element, and their set is denoted by $\mathrm{MIrr}(L)$. The \defn{length} of a chain is its number of elements minus one, and the length of $L$, denoted by $\ell(L)$, is the maximal length of a chain of $L$. A lattice $L$ is \defn{join-extremal} if $\ell(L)=|\mathrm{JIrr}(L)|$, \defn{meet-extremal} if $\ell(L)=|\mathrm{MIrr}(L)|$, and \defn{extremal} if it is both join-extremal and meet-extremal \cite{MarkowskyExtremal}.

An \defn{edge-labeling} of a poset assigns to each cover relation a label
in a set. If the set of labels is totally ordered, such an edge-labeling is
an \defn{EL-labeling} if, in every interval $[x,y]$, there is a unique maximal chain whose labels are strictly increasing, and its label sequence is lexicographically smallest among those of all maximal chains in $[x,y]$. A poset is \defn{EL-shellable} if it admits an EL-labeling, and this property implies shellability \cite{bjorner1983lexicographically,bjorner1996shellable}.
For lattices, EL-shellability is itself implied by left modularity \cite{LiuLeftmodular}. An element $m\in L$ is \defn{left modular} if, for every $x,y\in L$ with $x\leq y$, we have $x\vee(m\wedge y)=(x\vee m)\wedge y$. A lattice $L$ is \defn{left modular} if it has a maximal chain consisting entirely of left-modular elements \cite{blass1997mobius}. 

For general lattices, extremality and left modularity are different properties, but remarkably they coincide for semidistributive lattices. M\"uhle asked whether every semidistributive EL-shellable lattice is left modular, and whether every such lattice is extremal \cite[Questions 4.4 and 4.5]{muhle2023extremality}. In \cite{segovia2025extremality}, the second author proved that a congruence uniform lattice (which is semidistributive) is shellable if and only if it is extremal, and asked whether the implication from shellability to extremality remains true for all semidistributive lattices. We answer these questions affirmatively.

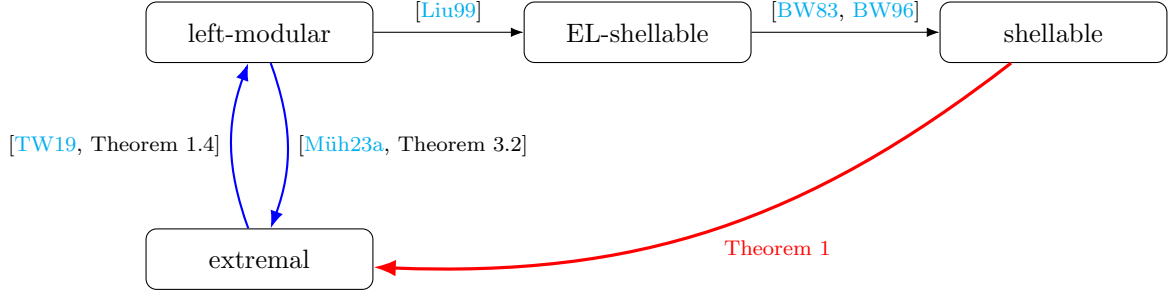
\begin{figure}[t]
\centering
\begin{tikzpicture}[
    every node/.style={font=\small},
    box/.style={
        draw,
        rounded corners,
        minimum width=3cm,
        minimum height=.8cm,
        align=center
    }
]

\node[box] (lm) at (0,0) {left-modular};
\node[box] (el)  at (5,0) {EL-shellable};
\node[box] (sh)  at (10.5,0) {shellable};
\node[box] (ext)  at (0,-3) {extremal};

\draw[-{Latex}]
(lm) -- node[above,font=\scriptsize]{\cite{LiuLeftmodular}} (el);

\draw[-{Latex}]
(el) --node[above,font=\scriptsize]{\cite{bjorner1983lexicographically,bjorner1996shellable}} (sh);

\draw[blue,thick,-{Latex}]
(ext) to[bend left=20]
node[left,text=black,font=\scriptsize]
{\cite[Theorem 1.4]{Thomas_2019}} (lm);

\draw[blue,thick,-{Latex}]
(lm) to[bend left=20]
node[right,text=black,font=\scriptsize]
{\cite[Theorem 3.2]{muhle2023extremality}} (ext);

\draw[red,very thick,-{Latex}]
(sh) to[bend left=20]
node[below right,font=\scriptsize,align=center]
{Theorem~\ref{thm:main}} (ext);

\end{tikzpicture}

\caption{The black implications are true for any lattice, whereas the thick blue and red implications require the additional assumption that the lattices are semidistributive.
The red implication is established in this paper.}
\label{fig:main-implications}
\end{figure}

\begin{theorem}\label{thm:main}
Let $L$ be a join-semidistributive lattice. If $L$ is shellable, then $L$ is join-extremal. Thus, if $L$ is semidistributive and shellable, then $L$ is extremal.
\end{theorem}

The converse of the first statement of Theorem~\ref{thm:main} does not hold:
Figure~\ref{fig:counterexample} shows a join-semidistributive and
join-extremal lattice that is not shellable.

Together with known results, Theorem \ref{thm:main} immediately yields the following
equivalence, summarized in Figure \ref{fig:main-implications}.

\begin{theorem}
Let $L$ be a semidistributive lattice. The following properties are equivalent:
\begin{itemize}
\item[$(i)$] $L$ is extremal.
\item[$(ii)$] $L$ is left modular.
\item[$(iii)$] $L$ is EL-shellable.
\item[$(iv)$] $L$ is shellable.
\end{itemize}
\end{theorem}

We start by recalling in Section \ref{sec:edge-label} a well-known edge-labeling by join-irreducibles for join-semidistributive lattices. Then, in Section \ref{sec:inclusionlabels}, we compare the sets of labels of different maximal chains. Finally, we prove Theorem \ref{thm:main} in Section \ref{sec:mainresult}.

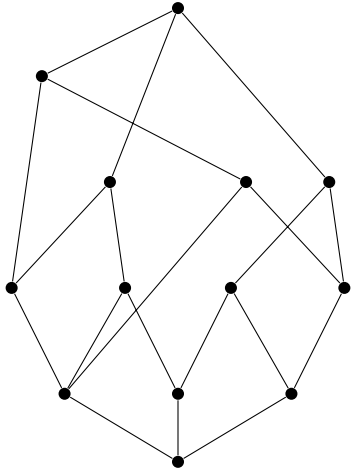
\begin{figure}[t]
\centering
\begin{tikzpicture}[
    scale=1,
    every node/.style={circle, fill=black, inner sep=1.6pt}
]

\node (0) at (0,0.5) {};

\node (11) at (-1.5,1.4) {};
\node (9)  at (0,1.4) {};
\node (12) at (1.5,1.4) {};

\node (8)  at (-2.2,2.8) {};
\node (5)  at (-0.7,2.8) {};
\node (6)  at (0.7,2.8) {};
\node (10) at (2.2,2.8) {};

\node (2) at (-0.9,4.2) {};
\node (7) at (0.9,4.2) {};
\node (3) at (2,4.2) {};

\node (4) at (-1.8,5.6) {};

\node (1) at (0,6.5) {};

\draw
(0) -- (9)
(0) -- (11)
(0) -- (12)

(9) -- (5)
(9) -- (6)

(11) -- (5)
(11) -- (7)
(11) -- (8)

(12) -- (6)
(12) -- (10)

(8) -- (2)
(8) -- (4)

(5) -- (2)

(6) -- (3)

(10) -- (3)
(10) -- (7)

(7) -- (4)

(2) -- (1)
(3) -- (1)
(4) -- (1);

\end{tikzpicture}
\caption{A join-semidistributive and join-extremal lattice that is not shellable.}
\label{fig:counterexample}
\end{figure}

\section{Edge-labeling by join-irreducibles}
\label{sec:edge-label}

Let $P$ be a poset. For $x,y\in P$, we say that $y$ \defn{covers} $x$, denoted by $x\lessdot y$, if $x<y$ and there is no $z\in P$ such that $x<z<y$.
For $x\in P$, let $I_P(x)=\{y\in P\mid y\leq x\}$ be the \defn{order ideal} of $P$ generated by $x$. 

We recall the following well-known fact about join-semidistributive lattices.

\begin{lemma}[{\cite[Section 5]{freese1995free}}]
\label{lem:canonicaljoinlabel}
Let $L$ be a join-semidistributive lattice and let $x\lessdot y$. Then $I_L(y)\setminus I_L(x)$ has a minimum element, which is denoted by $\gamma_J(x\lessdot y)$.
Then $\gamma_J(x\lessdot y)\in \mathrm{JIrr}(L)$ and $x\vee\gamma_J(x\lessdot y)=y$.
\end{lemma}

\begin{proof}
Let $u$ be a minimal element of $I_L(y)\setminus I_L(x)$. Since $u\leq y$ and $u\nleq x$, we have $x<x\vee u\leq y$.
As $x\lessdot y$, it follows that $x\vee u=y$.

Suppose that $u$ and $v$ are two minimal elements of $I_L(y)\setminus I_L(x)$. Then $x\vee u=x\vee v=y$.
By join-semidistributivity, $x\vee(u\wedge v)=y$.
Thus $u\wedge v\in I_L(y)\setminus I_L(x)$. The minimality of $u$ and $v$ gives $u=u\wedge v=v$. This proves that $I_L(y)\setminus I_L(x)$ has a minimum element.

Set $p=\gamma_J(x\lessdot y)$. We already proved that $x\vee p=y$. Suppose that $p=a\vee b$ with $a,b<p$, which is equivalent to $p$ not being a join-irreducible since $L$ is finite. Since $p\nleq x$, at least one of $a$ and $b$ is not below $x$; assume that $a\nleq x$. Then $a\in I_L(y)\setminus I_L(x)$ and $a<p$, contradicting the minimality of $p$. Hence $p$ is join-irreducible.
\end{proof}

As proved in Lemma \ref{lem:canonicaljoinlabel}, a join-semidistributive lattice $L$ has an edge-labeling $\gamma_J$ by join-irreducibles defined by $\gamma_J(x\lessdot y) = \min \Big(I_L(y)\setminus I_L(x) \Big)$.

For a maximal chain $F:\quad x_0\lessdot x_1\lessdot\cdots\lessdot x_n$ of a join-semidistributive lattice $L$,
define
\[
\gamma(F)=\left\{\gamma_J(x_{i-1}\lessdot x_i)\;\middle|\;1\leq i\leq n\right\}.
\]

\begin{lemma}\label{lem:distinctlabels}
The labels of $\gamma_J$ along a maximal chain of a join-semidistributive lattice are pairwise distinct. In particular, the length of a maximal chain $F$ is $|\gamma(F)|$.
\end{lemma}

\begin{proof}
Let $x_0\lessdot x_1\lessdot\cdots\lessdot x_n$
be a maximal chain and suppose that $i<j$. We have \mbox{$\gamma_J(x_{i-1}\lessdot x_i)\leq x_i\leq x_{j-1}$},
while $\gamma_J(x_{j-1}\lessdot x_j)\nleq x_{j-1}$.
Thus these two labels are different.
\end{proof}

\section{An inclusion of edge-label sets}
\label{sec:inclusionlabels}

Let $F$ and $G$ be maximal chains of $L$. We write $F\rightarrow G$ if $|F\cap G|=|G|-1$. The proof of the following result is illustrated in Figure \ref{fig:perspectivecovers}.

\begin{figure}[H]
\centering
\begin{tikzpicture}[scale=1.5]
\draw[blue] (0,0)--(1,2)--(0,4);
\draw[red] (-0.8,0.8)--(0,0);
\draw[red] (-0.8,3.2)--(0,4);
\draw[dashed,red] (-0.8,3.2)--(-0.8,0.8);

\node[red] at (-1.5,2) {$F$};
\node[blue] at (1.7,2) {$G$};

\draw (0,0) node[below]{$a=y_0$};
\draw (0,0) node{$\bullet$};
\draw (0,4) node[above]{$b=y_k$};
\draw (0,4) node{$\bullet$};
\draw (-0.8,0.8) node[left]{$y_1$};
\draw (-0.8,0.8) node{$\bullet$};
\draw (-0.8,3.2) node[left]{$y_{k-1}$};
\draw (-0.8,3.2) node{$\bullet$};
\draw (1,2) node[right]{$c$};
\draw (1,2) node{$\bullet$};
				
\draw (0.5,1) node[below right]{$p$};
\draw (0.5,3) node[above right]{$q$};
\draw (-0.4,0.4) node[below left]{$q'$};
\draw (-0.4,3.6) node[above left]{$p'$}; 
\end{tikzpicture}
\caption{Illustration of the proof of Lemma \ref{lem:gammarenversearrow}.}
\label{fig:perspectivecovers}
\end{figure}
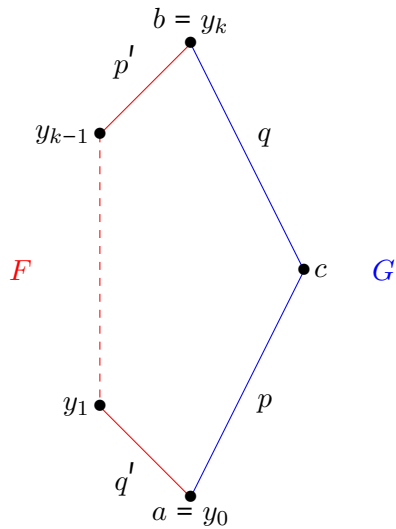

\begin{lemma}\label{lem:gammarenversearrow}
Let $F$ and $G$ be maximal chains of a join-semidistributive lattice $L$. If $F\rightarrow G$, then $\gamma(G)\subseteq\gamma(F)$.
\end{lemma}

\begin{proof}
Suppose that $F\rightarrow G$, meaning there exists a unique $c\in G\setminus F$ and $F\cap G=G\setminus\{c\}$. Let $a,b\in G$ be such that $a\lessdot c\lessdot b$.
Since $a,b\in F$, the elements between $a$ and $b$ along $F$ form a chain $a=y_0\lessdot y_1\lessdot\cdots\lessdot y_k=b$.
Every cover relation of $G$, except $a\lessdot c$ and $c\lessdot b$, is also a cover relation of $F$. It is therefore enough to show that the labels of these two covers occur along the chain $y_0\lessdot y_1\lessdot\cdots\lessdot y_k$.
To prove this, we now prove that $\gamma_J(y_{k-1}\lessdot b)=\gamma_J(a\lessdot c)$ and 
$\gamma_J(a\lessdot y_1)=\gamma_J(c\lessdot b)$.

Set $p=\gamma_J(a\lessdot c)$ and $p'=\gamma_J(y_{k-1}\lessdot b)$.
By Lemma \ref{lem:canonicaljoinlabel}, we have $a\vee p=c$. Suppose that $p\leq y_{k-1}$. Then $c=a\vee p\leq y_{k-1}<b$.
Since $c\lessdot b$, this forces $y_{k-1}=c$, which is impossible because $c\notin F$. Hence $p\in I_L(b)\setminus I_L(y_{k-1})$, so $p'\leq p$. Since $p'\leq p\leq c$ and $p'\nleq a$ (because $p'\nleq y_{k-1}$), we have $p'\in I_L(c)\setminus I_L(a)$. Then $p\leq p'$. Hence $p=p'$.

Now set $q=\gamma_J(c\lessdot b)$ and $q'=\gamma_J(a\lessdot y_1)$.
By Lemma \ref{lem:canonicaljoinlabel}, we have $a\vee q'=y_1$. Suppose that $q'\leq c$. Then $a<  y_1=a\vee q'\leq c$.
Since $a\lessdot c$, this forces $y_1=c$, again contradicting $c\notin F$. Hence $q'\in I_L(b)\setminus I_L(c)$,
so $q\leq q'$. Since $q\leq q'\leq y_1$ and $q\nleq a$ (because $q\nleq c$), we have $q\in I_L(y_1)\setminus I_L(a)$.
Then $q'\leq q$. Hence $q=q'$.
\end{proof}

\begin{remark}
In the proof of Lemma \ref{lem:gammarenversearrow}, $p'=p$ and 
$q'=q$ also follow from \cite[Lemma 3.3]{Muhle_2022} using the fact that the covers $y_{k-1}\lessdot b$ and $a\lessdot c$ are perspective, as are $a\lessdot y_1$ and $c\lessdot b$.
\end{remark}

\section{Proof of the main result}
\label{sec:mainresult}

The following result follows immediately from the definition of a shelling.

\begin{lemma}\label{lem:arrowadjacency}
If $F_1,F_2,\dots,F_k$ is a shelling of $\Delta(L)$, then for every $2\leq j\leq k$, there exists $i<j$ such that $F_i\rightarrow F_j$.
\end{lemma}

We can now prove our main result.

\begin{proof}[Proof of Theorem \ref{thm:main}]
Suppose that $L$ is join-semidistributive and shellable. Let us prove that $L$ is join-extremal.
The result is immediate if $L$ is a chain, so we assume that $L$ has at least two maximal chains. Let $F_1,F_2,\dots,F_k$
be a shelling of $\Delta(L)$. By Lemma \ref{lem:arrowadjacency}, for every $j>1$ there exists $i<j$ such that $F_i\rightarrow F_j$. Lemma \ref{lem:gammarenversearrow} then gives $\gamma(F_j)\subseteq\gamma(F_i)$.
By induction on $j$, it follows that $\gamma(F_j)\subseteq\gamma(F_1)$
for every $j\in \{1,2,\dots,k\}$. 

Let $p\in\mathrm{JIrr}(L)$. The cover $p_*\lessdot p$ can be extended to a maximal chain $F_j$. Since \mbox{$I_L(p)\setminus I_L(p_*)=\{p\}$},
we have $\gamma_J(p_*\lessdot p)=p$. Thus $p\in\gamma(F_j)$, and consequently $p\in\gamma(F_1)$. This proves $\mathrm{JIrr}(L)\subseteq\gamma(F_1)$.
The reverse inclusion follows from Lemma \ref{lem:canonicaljoinlabel}, and hence $\gamma(F_1)=\mathrm{JIrr}(L)$.
Then by Lemma \ref{lem:distinctlabels}, the length of $F_1$ is $|\mathrm{JIrr}(L)|$. 
Moreover, for any maximal chain $F$, the same lemma shows that its length is $|\gamma(F)|\leq |\mathrm{JIrr}(L)|$.
Therefore $\ell(L)=|\mathrm{JIrr}(L)|$, so $L$ is join-extremal.

Now assume that $L$ is semidistributive and shellable. We just proved that $L$ is join-extremal. For a semidistributive lattice, it is known that $|\mathrm{JIrr}(L)|=|\mathrm{MIrr}(L)|$ \cite[Section 5]{freese1995free}.
Thus $\ell(L)=|\mathrm{JIrr}(L)|=|\mathrm{MIrr}(L)|$, and $L$ is extremal.
\end{proof}

\section*{Acknowledgements}

The authors acknowledge Gewu Intelligence Lab for providing access to its collaborative research infrastructure and for helpful advice regarding the publication of this work.

\section*{Disclosure of automated assistance}

OpenAI GPT-5.6 Sol was used in proof exploration and language
editing. The arguments in this paper are self-contained
and the authors assume full responsibility for the final manuscript.

\bibliographystyle{alpha}
\bibliography{references}

\end{document}